\documentclass[12pt]{amsart}
\usepackage{amsmath}
\usepackage{amssymb}
\usepackage{ifthen}
\usepackage{graphicx}
\usepackage{float}
\usepackage{caption}
\usepackage{subcaption}
\usepackage{cite}
\usepackage{amsfonts}
\usepackage{amscd}
\usepackage{amsxtra}
\usepackage{color}

\newtheorem{theorem}{Theorem}[section]
\newtheorem{lemma}[theorem]{Lemma}
\newtheorem{corollary}[theorem]{Corollary}

\theoremstyle{definition}
\newtheorem{remark}[theorem]{Remark}

\newtheorem{example}[theorem]{Example}

\numberwithin{equation}{section}

\def\be{\begin{equation}}
\def\ee{\end{equation}}

\newcounter{alphabet}
\newcounter{tmp}
\newenvironment{Thm}[1][]{\refstepcounter{alphabet}%
\bigskip%
\noindent%
{\bf Theorem \Alph{alphabet}}%
\ifthenelse{\equal{#1}{}}{}{ (#1)}%
{\bf .} \itshape}{\vskip 8pt}
\makeatletter
\newcommand{\Ref}[1]{\@ifundefined{r@#1}{}{\setcounter{tmp}{\ref{#1}}\Alph{tmp}}}
\makeatother
\begin{document}

\title{Successive coefficients for spirallike and related functions}
\author[V. Arora]{Vibhuti Arora}
\address{Discipline of Mathematics \\ Indian Institute of Technology Indore\\
Simrol, Khandwa Road\\
Indore 453 552, India}
\email{vibhutiarora1991@gmail.com}

\author[S. Ponnusamy]{Saminathan Ponnusamy}
\address{Department of Mathematics \\ Indian Institute of Technology Madras\\
Chennai 600 036, India}
\email{samy@iitm.ac.in}
\author[S. K. Sahoo]{Swadesh Kumar Sahoo}
\address{Discipline of Mathematics \\ Indian Institute of Technology Indore\\
Simrol, Khandwa Road\\
Indore 453 552, India}
\email{swadesh@iiti.ac.in}

\begin{abstract}
We consider the family of all analytic and univalent functions in the unit disk of the form $f(z)=z+a_2z^2+a_3z^3+\cdots$. Our objective in this paper is to estimate the difference of the moduli of successive coefficients, that is $\big | |a_{n+1}|-|a_n|\big |$, for $f$ belonging to the family of $\gamma$-spirallike functions of order $\alpha$. Our particular results include the case of starlike and convex functions of order $\alpha$
 and other related class of functions.
\end{abstract}

\maketitle

\section{Introduction and statement of a main result}\label{2sec1}
Let us denote the family of all meromorphic functions $f$ with no poles in the unit disk $\mathbb{D}:=\{z \in \mathbb{C}:|z|<1\}$ of the form
\begin{equation}\label{deq1}
f(z)=z+a_2z^2+a_3z^3+\cdots
\end{equation}
by $\mathcal{A}$. Clearly, functions in  $\mathcal{A}$ are analytic in $\mathbb{D}$ and
the set of all univalent functions $f \in \mathcal{A}$ is denoted by $\mathcal{S}$. Functions in $\mathcal{S}$ are of interest because they appear in the Riemann mapping theorem and several other situation in many different contexts. For background knowledge on these settings we refer to the standard books \cite{AvWir-09, Dur83, Gol46, Goo83, Mil77}.

One of the popular necessary conditions for a function $f$ of the form \eqref{deq1} to be in $\mathcal{S}$
is the sharp inequality $|a_n|\leq n$ for $n\geq 2$, which was first conjectured by Bieberbach in 1916 and
proved by de Branges in 1985 (\cite{DeB1}). On the other hand, the problem of estimating sharp bound for successive coefficients,
namely, $\big | |a_{n+1}|-|a_n| \big | $, is also an
interesting necessary condition for a function to be in $\mathcal{S}$. This problem was first studied by Goluzin \cite{Gol46}
with an idea to solve the Bieberbach conjecture. Several results are known in this direction. For example, Hamilton \cite{Ham80}
proved that $\displaystyle \overline{\lim}_{n\rightarrow \infty}\big | |a_{n+1}|-|a_n| \big | \leq 1$. Prior to this paper,
Hayman \cite{Hay63} proved in 1963 that
\begin{equation}\label{deq5}
\big | |a_{n+1}|-|a_n| \big | \leq A, \quad n=1,2,3,\dots,
\end{equation}
where $A\geq 1$ is an absolute constant, for functions $f$ in $\mathcal{S}$ of the form \eqref{deq1}.
Milin \cite{Mil68,Mil77} found a simpler approach, which led to the bound $A\leq 9$ and  Ilina \cite{Ili68}
improved this to $A\leq 4.26$. It is still an open problem to find the minimal value of $A$
which works for all $f \in \mathcal{S}$, however, the best known bound as of now is $3.61$ which is due to
Grinspan \cite{Gri76} (see also \cite{Mil77}). The fact that $A$ in \eqref{deq5} cannot be replaced by $1$
may be seen from the work of \cite{SS43}. On the other hand, sharp bound is known only for $n=2$ (see \cite[Theorem~3.11]{Dur83}), namely,
$$
-1\leq|a_3|-|a_2|\leq 1.029\ldots.
$$
Since Schaeffer and Spencer \cite{SS43} showed that for each $n \geq 2$ there corresponds an odd function
$h(z)=z+a_3z^3+\cdots$ in $\mathcal{S}$ with all of its coefficients real such that $|a_{2n+1}(h)|>1$, it is also
clear that the constant $A$ in \eqref{deq5} must be greater than $1$ for odd functions in the  class $\mathcal{S}$.
Note that for the Koebe function $k(z)=z/(1-z)^2$ and its rotation $e^{-i\theta}k(e^{i\theta}z)$, we have  $\big | |a_{n+1}|-|a_n| \big | =1$ for $n\geq 1$.

Denote by ${\mathcal S}^*$, the class $\mathcal S$ of functions $f$ such that $f(\mathbb{D})$ is starlike with respect to the origin.
Concerning the class ${\mathcal S}^*$,  Leung \cite{Leu78} (see also \cite{LS17}) in 1978 has proved that $A=1$ for starlike functions that was
first conjectured by Pommerenke in \cite{POM71}. More precisely, we have

%\medskip

\begin{Thm}\label{ThA}
{\rm \cite{Leu78}}
For every $f\in \mathcal{S}^*$ given by \eqref{deq1}, we have
$$\big | |a_{n+1}|-|a_n| \big | \leq 1, \quad n=1,2,3,\ldots.
$$
Equality occurs for fixed $n$ only for the function
$$
\frac{z}{(1-\gamma z)(1-\zeta z)}
$$
for some $\gamma$ and $\zeta$ with $|\gamma|=|\zeta|=1$.
\end{Thm}

We remark that, as an application of triangular inequality, Theorem \Ref{ThA} leads to $|a_n|\leq n$ for $n\geq 2$ which is the well known
coefficient inequality for starlike functions.
This is one of reasons for studying the successive coefficients problem in the univalent function theory.
From the above discussion, we understand the importance of finding the minimal value of $A$ for functions to be in $\mathcal{S}$.
Later, the problem of finding the minimal value of $A$ was considered for certain other subfamilies of univalent functions
such as convex, close-to-convex, and spirallike functions. Among other things, Hamilton in \cite{Ham80} has shown some
bound for successive coefficients for spirallike functions and for the class of starlike
functions of non-positive order. For convex functions, recently Li and Sugawa \cite{LS17} obtained the sharp upper
bound which is $|a_{n+1}|-|a_n|\leq1/(n+1)$ for $n \geq 2$, and for $n=2,3$ sharp lower bounds are
$1/2$ and $1/3$, respectively. For $n\geq 4$, it is still an open problem to find the best lower
bound for convex functions. These information clearly shows the level of difficulty in determining the bound on the
successive coefficients problem.

Our objective in this paper is to obtain results related to successive coefficients for starlike functions
of order $\alpha$, convex functions of order $\alpha$, spirallike functions and functions in the close-to-convex family.

To state our first result we need to introduce the following definitions:
The family $\mathcal{S}_\gamma (\alpha )$ of  $\gamma$-spirallike functions of order $\alpha$ is defined by
$$ \mathcal{S}_\gamma (\alpha)
= \left \{f\in {\mathcal A}: \, {\rm Re}  \left ( e^{-i\gamma}\frac{zf'(z)}{f(z)}\right )>\alpha \cos \gamma\,
 \mbox{ for }  z\in \mathbb{D}\right\},
$$
where $ \alpha \in [0,1)$ and $\gamma\in (-\pi/2, \pi/2)$.
Each function in $\mathcal{S}_\gamma(\alpha )$ is univalent in $\mathbb{D}$ (see \cite{Lib67}).
Clearly, $\mathcal{S}_\gamma (\alpha )\subset \mathcal{S}_\gamma (0)\subset \mathcal{S}$
whenever $0\leq \alpha <1$. Functions in $\mathcal{S}_\gamma(0)$  are called \textit{$\gamma$-spirallike}, but they do not necessarily
belong to the starlike family $\mathcal{S}^*$.
The class $\mathcal{S}_\gamma (0)$ was introduced by ${\rm \check{S}}$pa${\rm\check{c}}$ek \cite{Spacek-33}
(see also \cite{Dur83}). Moreover, $\mathcal{S}_0 (\alpha)=:\mathcal{S}^*{(\alpha)}$ is the usual class of starlike functions of order $\alpha$,
and $\mathcal{S}^*(0)=\mathcal{S}^*$. The class $\mathcal{S}^*{(\alpha)}$ is meaningful even if $\alpha <0$, although univalency will be destroyed
in this situation.

A function $f \in \mathcal{A}$ is called convex of order $\alpha$, denoted by $\mathcal{C}(\alpha)$ if and only if, for some $\alpha \in [ 0,1),$
$zf'(z)$ belongs to $\mathcal{S}^*{(\alpha)}$; i.e.
\begin{equation}\label{deq4}
{\rm Re}\left (1+\frac{zf''(z)}{f'(z)}\right)>\alpha \quad \mbox{ for $z\in{\mathbb D}$}.
\end{equation}
If $\alpha=0,$ the inequality \eqref{deq4} is equivalent to the definition of a convex function, i.e. $f$
maps $\mathbb{D}$ onto a convex domain. We set $\mathcal{C}(0)=\mathcal{C}$. It is well-known that $\mathcal C$ is a proper subset of
$\mathcal{S}^*{(1/2)}$.

We state our first result which shows that Theorem \Ref{ThA} continues to hold for $\gamma$-spirallike functions. More generally,
as a generalization and the extension of  Leung's result, we prove the following result whose proof will be presented in Section \ref{2sec4}.

\begin{theorem}\label{2thm1}
For every $f \in \mathcal{S}_\gamma (\alpha )$ of the form \eqref{deq1},
$$ \big | |a_{n+1}|-|a_n| \big | \leq \exp(-M\alpha \cos \gamma )
$$
for some absolute constant $M>0$ and \mbox{for $n \ge 2$}.
\end{theorem}

Note that for $\alpha =0$, the above theorem extend the result of  Leung \cite{Leu78} from starlike to $\gamma$-spirallike functions
and hence Theorem \ref{2thm1} contains the result of Hamilton \cite{Ham80}. For a ready reference, we recall it here.
However, in this paper, we get his result as a consequence of a general result with an alternate proof.

\begin{corollary}\label{2thm2}
Let $f \in \mathcal{S}_\gamma (0)$ for some $|\gamma|<\pi/2$, and be of the form \eqref{deq1}. Then
$$\big | |a_{n+1}|-|a_n| \big |  \leq 1 ~\mbox{ for $n \ge 2$}.
$$
\end{corollary}

\begin{remark}
In Theorem \ref{2thm4}, we see that Theorem~\Ref{ThA} and Corollary~\ref{2thm2} continue to hold for functions that are not necessarily
starlike but is close-to-convex. At this place it is worth pointing out that there are functions that are $\gamma$-spirallike
but not close-to-convex. It is also equally true that there exist close-to-convex functions but are not $\gamma$-spirallike.
Theorem \ref{2thm4} is supplementary for this reasoning.
\end{remark}

The paper is organized as follows. Section \ref{2sec2} deals with definitions of classes of functions and statements of main results. In Section \ref{2sec3}, we state and prove a lemma which will be used in the proof of our main results in Section \ref{2sec4}.

%%%%%%%%%%%%%%%%%%%%%%%%%%%%%%%%%%%%%%%%%%%%%
%%%%%%%%%%%%%%%%%%%%%%% section 2 %%%%%%%%%%%%%%%%%%%%%%
%%%%%%%%%%%%%%%%%%%%%%%%%%%%%%%%%%%%%%%%%%%%%%%%%%%%%

\section{Definitions and further results}\label{2sec2}

We consider another family of functions that includes
the class of convex functions as a proper subfamily. For
$-\pi/2<\gamma<\pi/2$, we say that $f\in {\mathcal C}_\gamma (\alpha) $
provided $f\in {\mathcal A}$ is locally univalent in $\mathbb{D}$
and  $zf'(z)$ belongs to $\mathcal{S}_\gamma (\alpha)$, i.e.
\be\label{pv-hire2-eq1}
{\rm Re } \left \{ e^{-i\gamma }\left ( 1+\frac{zf''(z)}{f'(z)}\right )\right \}>\alpha \cos \gamma, \quad
z\in\mathbb{D}.
\ee
We may set ${\mathcal C}_\gamma (0) =:{\mathcal C}_\gamma$  and observe that the class
${\mathcal C}_0(\alpha)=:{\mathcal C}(\alpha)$ consists of the normalized
convex functions of order $\alpha$. For general values of $\gamma $
$(|\gamma|<\pi/2)$, a function in ${\mathcal C}_\gamma (0)$ need not be
univalent in $\mathbb{D}$. For example, the function
$f(z)=i(1-z)^i-i$ is known to belong to
${\mathcal C}_{\pi/4}\backslash {\mathcal S}$. Robertson
\cite{Robertson-69} showed  that $f\in\mathcal{C}_{\gamma}$ is
univalent if $0<\cos \gamma\leq 0.2315\cdots$. Finally, Pfaltzgraff \cite{Pfaltzgraff} has shown that
$f\in\mathcal{C}_{\gamma}$ is univalent whenever $0<\cos \gamma\leq
1/2$. This settles the improvement of range of $\gamma$ for which $f\in\mathcal{C}_{\gamma}$
is univalent. On the other hand,  in \cite{SinghChic-77} it was also shown that
functions in ${\mathcal C}_\gamma$ which satisfy $f''(0)=0$ are
univalent for all real values of $\gamma$ with $|\gamma|<\pi /2$.
For a general reference about these special classes we refer to
\cite{Goo83}.

\begin{Thm} \label{ThB}
{\rm \cite{LS17}}
For every $f \in {\mathcal C}:={\mathcal C}(0)$ of the form \eqref{deq1}, the following inequality holds
$$
|a_{n+1}|-|a_n| \leq \frac{1}{n+1}
$$
for $n\geq 2$, and the extremal function is given by
$$
L_{\phi}(z)=\cfrac{1}{e^{i\phi}-e^{-i\phi}} ~\log\left (\frac{1-e^{-i\phi}z}{1-e^{i\phi}z} \right )
$$
for $\phi=\pi/n$, where a principal branch of logarithm is chosen.
\end{Thm}

A straightforward application of Theorem \ref{2thm1} yields the following generalization of Theorem \Ref{ThB}
for convex functions of order $\alpha$ and also for locally univalent functions that are not necessarily univalent in the unit
disk $\mathbb D$.

\begin{corollary}\label{2thm1:coro1}
Suppose that $f\in {\mathcal C}_\gamma (\alpha)$ for some $\alpha \in [0,1)$ and $-\pi/2<\gamma<\pi/2$. Then we have
$$|a_{n+1}|-|a_n|\leq\cfrac{\exp(-M\alpha \cos \gamma)}{n+1}
$$
for some absolute constant $M>0$. In particular, we have
\begin{enumerate}
\item [(1)] For $f\in {\mathcal C}_\gamma (0)$,
$$|a_{n+1}|-|a_n|\leq \frac{1}{n+1}.
$$

\item [(2)] For $f\in {\mathcal C}(\alpha)$ we have
$$
|a_{n+1}|-|a_n|\leq\cfrac{\exp(-M\alpha)}{n+1}
$$
for some absolute constant $M>0$.
\end{enumerate}
\end{corollary}
\begin{proof}
By the classical Alexander theorem, $f(z)=z+\sum_{n=2}^{\infty} a_n z^n$ belongs to ${\mathcal C}_\gamma (\alpha)$ if and only if $zf'(z)=z+\sum_{n=2}^{\infty} b_n z^n$ is ${\mathcal S}_\gamma (\alpha)$ and clearly, $b_n=n a_n.$ Thus, by Theorem \ref{2thm1}, we have
$$
(n+1)|a_{n+1}|-n|a_n|=|b_{n+1}|-|b_n|\leq \exp(-M\alpha \cos \gamma).
$$
This gives,
$$
 |a_{n+1}|-|a_n| \leq |a_{n+1}|-\cfrac{n}{n+1}|a_n| \leq \cfrac{\exp(-M\alpha \cos \gamma)}{n+1}\,.
 $$
The proof of the corollary is complete.
\end{proof}

We would like to remark that Hamilton generalized Leung's result to the case of starlike functions of non-positive order and proved the following:

%\medskip

\begin{Thm} \label{ThC}
{\rm \cite{Ham80}}
For a function $f(z) \in {\mathcal S}^*(\alpha )$ for some $\alpha \leq 0,$
$$
\big | |a_{n+1}|-|a_n| \big |  \leq \cfrac{\Gamma(1-2\alpha +n)}{\Gamma(1-2\alpha )\Gamma(n+1)}\,.
$$
Equality holds for the function $f(z)=z(1-z)^{2(\alpha -1)}$.
\end{Thm}

Let $f\in\mathcal{A}$ be locally univalent. Then, according to Kaplan's theorem, it follows that
$f$ is {\em close-to-convex} if and only if for each $r~(0<r<1)$
and for each pair of real numbers $\theta_1$ and $\theta_2$ with $\theta_1<\theta_2$,
$$
 \int_{\theta_1}^{\theta_2} {\rm Re}\left(1+\frac{zf''(z)}{f'(z)}\right)\,d\theta>-\pi,\quad z=re^{i\theta}.
$$
If a locally univalent analytic function $f$ defined in $\mathbb{D}$ satisfies
$$
 {\rm Re}\left(1+\frac{zf''(z)}{f'(z)}\right)>-\frac{1}{2}  ~\mbox{ for $z \in \mathbb{D}$},
$$
then by the Kaplan characterization it follows easily that $f$ is close-to-convex in $\mathbb{D}$,
and hence $f$ is univalent in $\mathbb{D}$.
This generates the following subclass of the class of close-to-convex (univalent) functions:
$$
\mathcal{C} (-1/2):=\left\{f\in \mathcal{A}:\,{\rm Re}\left (1+ \frac{zf''(z)}{f'(z)}\right )>-\frac{1}{2}  ~\mbox{ for $z \in \mathbb{D}$}\right\}.
$$
This class of functions is also studied recently by the authors in \cite{AS17}, and others in different contexts; for instance see \cite{LP17,MYLP14,PSY14} and references therein. Functions in $\mathcal{C} (-1/2)$ are not necessarily starlike but is convex in some
direction as the function
\begin{equation}\label{2condirection}
f(z)=\cfrac{z-(z^2/2)}{(1-z)^2}
\end{equation}
shows. Note that  
$$
{\rm Re}\left (1+ \frac{zf''(z)}{f'(z)}\right ) ={\rm Re} \left(\cfrac{1+2z}{1-z}\right)>-\frac{1}{2} ~\mbox{ for $z \in \mathbb{D}$}
$$
and thus  $f\in \mathcal{C} (-1/2)$, but not starlike in $\mathbb{D}$.

% because
%$$
%{\rm Re}\left (\frac{zf'(z)}{f(z)}\right ) ={\rm Re}\left(\cfrac{-2}{2-z}+\cfrac{2}{1-z}\right)<0.
%$$
%

\begin{theorem}\label{2thm4}
Let $f \in \mathcal{C} (-1/2)$. Then
$$
|a_{n+1}|-|a_n|\leq1.
$$
\end{theorem}

The following result is an immediate consequence of Theorem \ref{2thm4} which solves the Robertson conjecture problem
for the class $\mathcal{C}(-1/2)$. It is worth pointing out that in 1966 Robertson \cite{Rob66} conjectured that the Bieberbach Conjecture could be strengthened to
$$\big | n|a_n|-m|a_m| \big | \le \big|n^2-m^2\big| \quad \mbox{ for all $m,n\ge 2$},
$$
however, two years latter Jenkins \cite{Jen68} showed that this inequality fails in the class $\mathcal{S}$.

\begin{theorem}\label{2thm5}
Let $f \in \mathcal{C}(-1/2)$. Then for $n>m$ we have
$$
\big | n|a_n|-m|a_m| \big |  \leq \frac{(n^2-m^2)+(n-m)}{2}=\frac{(n-m)(n+m+1)}{2}.
$$
Equality holds for $f(z)=(z-(z^2/2))/(1-z)^2$.
\end{theorem}

%%%%%%%%%%%%%%%%%%%%%%%%%%%%%%%%%%%%%%%
%%%%%%%%%%% section 3 %%%%%%%%%%%%%%%%%%%
%%%%%%%%%%%%%%%%%%%%%%%%%%%%%%%%%%%%%%%%%%
%
\section{Preliminary result}\label{2sec3}

The following lemma plays a crucial role in the proof of our main results.

%%%%%%%%%%%%%%lemma 1%%%%%%%%%%%%555555

\begin{lemma}\label{2lemma1}
Let $\varphi(z)=1+\sum_{n=1}^{\infty} c_nz^n$ be analytic in $\mathbb{D}$ such that ${\rm Re}\,\varphi(z)>\alpha$
in $\mathbb{D}$ for some $\alpha<1$.
Suppose that $\psi(z)=e^{i\gamma}\sum_{n=1}^\infty \lambda_n c_n z^n$ is analytic in $\mathbb{D}$,
where $\lambda_n \geq 0$ and ${\rm Re}\,\psi(z) \le M$ for some $M>0$.
Then we have the inequality
$$
\cos\gamma\sum_{n=1}^\infty \lambda_n |c_n|^2 \leq 2M(1-\alpha).
$$
\end{lemma}
\begin{proof}
Let us first prove the result for $\alpha=0$. Consider the identity
$$4({\rm Re}\,\varphi)({\rm Re}\,\psi)=(\varphi+\overline{\varphi})(\psi+\overline{\psi})
=(\varphi\psi+\varphi\overline{\psi})+\overline{(\varphi\psi+\varphi\overline{\psi})}
$$
so that
\begin{equation}\label{lem1-eq1}
4\int_{|z|=r}({\rm Re}\,\varphi)({\rm Re}\,\psi)\,d\theta =
2 {\rm Re}\left (\int_{|z|=r}\varphi(z)\overline{\psi(z)}\,d\theta \right ),
\end{equation}
since (with $z=re^{i\theta}$)
\begin{equation}\label{lem1-eq2}
\int_0^{2\pi}\varphi(z)\psi(z)\,d\theta
=\int_{|z|=r}\varphi(z)\psi(z)\,\frac{dz}{iz}=0,
\end{equation}
by the Cauchy integral formula and the fact that $\psi(0)=0$.
Using the power series representation of $\varphi(z)$ and $\psi(z)$,  it follows that (since $\overline{z}=r^2/z$ on $|z|=r$)
\begin{eqnarray}\label{lem1-eq3}
\int_{|z|=r}\varphi(z)\overline{\psi(z)}\,d\theta & = &e^{-i\gamma}\int_{|z|=r}
\left [1+\sum_{n=1}^\infty c_nz^n\right]\left[\sum_{n=1}^\infty\overline{c_n}\lambda_n\frac{r^{2n}}{z^n}\right]\frac{dz}{iz} \nonumber\\
&=&2\pi e^{-i\gamma}\sum_{n=1}^\infty \lambda_n |c_n|^2  r^{2n}.
\end{eqnarray}
By \eqref{lem1-eq2}, \eqref{lem1-eq3} and the assumption that  ${\rm Re}\,\psi(z) \le M$ for some $M>0$, the identity \eqref{lem1-eq1} reduces to
$$
4\pi \cos \gamma\sum_{n=1}^\infty \lambda_n |c_n|^2 r^{2n}=4\int_0^{2\pi} ({\rm Re}\,\varphi(z))({\rm Re}\,\psi(z))\,d\theta\le 4M \int_0^{2\pi} {\rm Re}\,\varphi(z)\,d\theta=8M\pi,
$$
where we have used the fact that
\begin{align*}
\frac{1}{2\pi}\int_0^{2\pi}{\rm Re}\,\varphi(z)\,d\theta
& =  \frac{1}{2\pi}\int_0^{2\pi} \frac{\varphi(z)+\overline{\varphi(z)}}{2}\,d\theta\\
& = \frac{1}{4\pi}\left[\int_{|z|=r}\varphi(z)\,\frac{dz}{iz}\,+\,\overline{\int_{|z|=r}\varphi(z)\,\frac{dz}{iz}}\,\right]\\
& = \frac{1}{4\pi}(2\pi+2\pi)=1.
\end{align*}
The desired result for the case $\alpha=0$ follows by letting $r\to 1^{-}$ in the last inequality.

Finally, for the general case, we first observe that ${\rm Re}\,\Phi(z)>0$, where
$$
\Phi(z)=\frac{\varphi(z)-\alpha}{1-\alpha}=1+\sum_{n=1}^\infty d_nz^n, \quad d_n=\frac{c_n}{1-\alpha}.
$$
Also, the given condition on $\psi$ gives ${\rm Re}\,\Psi(z)\le \frac{M}{1-\alpha},$ where
$$
\Psi(z)=e^{i\gamma}\sum_{n=1}^\infty\lambda_n d_n z^n=\frac{1}{1-\alpha}\left ( e^{i\gamma}\sum_{n=1}^\infty \lambda_nc_nz^n\right ) =
\frac{1}{1-\alpha}\psi (z).
$$
Applying the previous arguments for the pair $(\Phi(z),\Psi(z))$, one obtains that
$$
\cos \gamma \sum_{n=1}^\infty \lambda_n |d_n|^2=\frac{\cos \gamma }{(1-\alpha)^2}\sum_{n=1}^\infty \lambda_n|c_n|^2\le \frac{2M}{1-\alpha}
$$
so that $\cos \gamma \sum_{n=1}^\infty \lambda_n|c_n|^2\le 2M(1-\alpha)$, as desired.
\end{proof}
%%%%%%%%%%%%%%%%%%%%%%%%%%%%%%%%%%%%%%%lemma 2%%%%%%%%%%%%%%%%%%%%

\begin{remark}
We remark that Lemma \ref{2lemma1} for $\gamma =0$ is obtained by
MacGregor\cite{MacGre69} (see also \cite{Leu78} and \cite[p.178, Lemma]{Dur83}).
\end{remark}

%%%%%%%%%%%%%lemma 3%%%%%%%%%%%%%%%%%%

%%%%%%%%%%%%%%%%%%%%%%%%%%%%%%%%%%%%%%%%%%%%%%%
%%%%%%%%%%%%%%%%%%%% section 4 %%%%%%%%%%%%%%%%%%
%%%%%%%%%%%%%%%%%%%%%%%%%%%%%%%%%%%%%%%%%%%%%%%%
%
\section{Proof of the main results}\label{2sec4}
We begin with the proof of Theorem \ref{2thm1}

\subsection{Proof of Theorem \ref{2thm1}}
Let  $f\in\mathcal{S}_\gamma (\alpha)$. Then by the definition, we may consider $\varphi$ by
$$ \frac{1}{\cos \gamma }\left [e^{-i\gamma} \cfrac{zf'(z)}{f(z)} +i\sin \gamma \right ]
= \varphi (z)
$$
so that
$$e^{-i\gamma} \left (\cfrac{zf'(z)}{f(z)} -1 \right )= \cos\gamma\, (\varphi (z) -1),
$$
where ${\rm Re}\,\{\varphi(z)\}>\alpha $ and $\varphi(z)=1+\sum_{n=1}^\infty c_nz^n$ is analytic in $\mathbb D$.
We may rewrite the last equation as
\begin{equation}\label{2thm1:eq0}
\cfrac{f'(z)}{f(z)}-\cfrac{1}{z}=
e^{i\gamma} \cos\gamma\, \sum_{n=1}^\infty c_nz^{n-1}
 \end{equation}
which by simple integration gives
\begin{equation}\label{2thm1:eq1}
\log\left (\cfrac{f(z)}{z}\right )= e^{i\gamma}  \cos\gamma\,  \sum_{n=1}^\infty \cfrac{c_nz^n}{n}\,,
\end{equation}
where we use the principal value of the logarithm such that $\log 1$=0. By the Taylor series expansion
of $\log (1-\xi z)$ and \eqref{2thm1:eq1}, we get
\begin{align}\label{2thm1:eq2}
\log {(1-\xi z)\cfrac{f(z)}{z}}
& =\sum_{n=1}^\infty \cfrac{C_n-\xi^n}{n}z^n
=\sum_{n=1}^\infty \alpha_n z^n,
 \end{align}
where  $C_n=e^{i\gamma} \cos\gamma\,c_n$ and
$$
\alpha_n=\frac{C_n-\xi^n}{n} =\frac{e^{i\gamma} \cos\gamma\,c_n-\xi^n}{n}.
$$
Also, for $|\xi|=1$, we have
\begin{align}\label{2thm1:eq3}
(1-\xi z)\cfrac{f(z)}{z}
&=\sum_{n=0}^\infty \beta_n z^n, \quad \beta_n=a_{n+1}-\xi a_n.
\end{align}
From \eqref{2thm1:eq2} and \eqref{2thm1:eq3}, it follows that
$$
\exp \Big(\sum_{n=1}^\infty \alpha_n z^n \Big)=\sum_{n=0}^\infty \beta_n z^n, ~\beta_0=1.
$$
Then, by the third Lebedev-Milin inequality (see \cite[p. 143]{Dur83}), we have
$$|\beta _n|^2\leq \exp\Bigg\{\sum_{k=1}^n\Bigg( k|\alpha _k|^2-\cfrac{1}{k}\Bigg) \Bigg\},
$$
or equivalently
\begin{equation}\label{2thm1:eq4}
|a_{n+1}-\xi a_n|^2\leq \exp\Bigg\{\sum_{k=1}^n\Bigg(\cfrac{|C_k -\xi^k |^2}{k}-\cfrac{1}{k}\Bigg) \Bigg\}\,.
\end{equation}
Now we consider
$$
\psi(z)=e^{i\gamma}\sum_{k=1}^n\cfrac{ c_k {z}^k }{k},
$$
and let $M$ be the maximum of ${\rm Re}\{\psi(z)\}$ on $|z|=1.$
Applying Lemma \ref{2lemma1} with $\lambda_k=1/k$ for $1 \leq k \leq n$ and $\lambda_k=0$ for $k>n$, we obtain
\begin{align*}
\sum_{k=1}^n\Bigg(\cfrac{|C_k -\xi^k |^2}{k}-\cfrac{1}{k}\Bigg)&=\cos^2\gamma\sum_{k=1}^n \cfrac{|c_k |^2}{k}-2\cos\gamma\,
\sum_{k=1}^n\cfrac{{\rm Re}( e^{i\gamma} c_k \overline{\xi}^k) }{k}\\
& \leq 2 M(1-\alpha)\cos\gamma -2\cos\gamma \, {\rm Re}\{\psi(\overline{\xi})\}.
\end{align*}
Choosing $\xi$ (say $\xi_0$) so that ${\rm Re}\{\psi(\overline{\xi_0})\}=M,$ we see that
 \begin{align*}
 \sum_{k=1}^n\Bigg(\cfrac{|C_k -\xi_0^k |^2}{k}-\cfrac{1}{k}\Bigg)&\leq 2 M(1-\alpha)\cos\gamma -2M\cos\gamma =-2M\alpha \cos\gamma.
\end{align*}
Hence from \eqref{2thm1:eq4}, $|a_{n+1}-\xi_0 a_n|\leq \exp(-M\alpha \cos\gamma)$ for some $\xi_0$ with $|\xi_0|=1.$ Since
$$
\big | |a_{n+1}|-|a_n| \big |  \leq |a_{n+1}-\xi_0 a_n|\leq \exp(-M\alpha \cos\gamma),
$$
the proof of our theorem is complete. \hfill{$\Box$}

\vspace{6pt}

Here we provide one example that associates to Theorem \ref{2thm1}.
%%%%%%%%%%%%%%%%%%%%%%%%%%%%%%%%%%%%%%%%%%%%%%

\begin{example}
Consider the function  $f(z):=f_{\gamma, \alpha}(z)=z/(1-z)^{\beta}$, where $\beta =2(1-\alpha)\cos \gamma$.
It is easy to check that $f\in\mathcal{S}_\gamma (\alpha)$,
$$
f(z)=z+\sum_{n=2}^\infty \frac{\Gamma(n+\beta)}{\Gamma(n+1)\Gamma(\beta )}z^n
~\mbox{ and }~  e^{-i\gamma}\cfrac{zf'(z)}{f(z)}=e^{-i\gamma}+  2(1-\alpha)\cos\gamma \cfrac{z}{1-z}.
$$
Again consider the function
$$
\varphi(z)=e^{-i\gamma}\cfrac{zf'(z)}{f(z)}= e^{-i\gamma}+2(1-\alpha)\cos\gamma \,\sum_{n=1}^\infty z^n.
$$
It is clear that ${\rm Re}\,(\varphi(z))>\alpha \cos\gamma$.  Now, if we adopt the proof of Lemma \ref{2lemma1} and Theorem \ref{2thm1} by assuming $\psi(z)= 2(1-\alpha) \sum_{n=1}^\infty  z^n$ and $\gamma =0$, then for $f\in\mathcal{S}^* (\alpha)$ we obtain
$$\big | |a_{n+1}|-|a_n| \big | \leq \exp(-\alpha M),\quad  M=2(1-\alpha)(\log n+1).
$$
\end{example}

%%%%%%%%%%%%%%%%%%%%%%%%%%%%%%%%%%%%%%%%%%%%%%%%%%%%%%%%%%%

%%%%%%%%%%%%%%%%%%%%%%%%%%%%%%%%%%%%%%%%%%%%%%%%%%%

\subsection{Proof of Theorem \ref{2thm4}}
Let $f\in\mathcal{C} (-1/2)$. Then  the function $g(z)=\sum_{n=1}^{\infty} b_nz^n=zf'(z)$, where $b_n=na_n$,  belongs to $\mathcal{S}^*(-1/2)$.
From Theorem \Ref{ThC}, we obtain that
\begin{equation}\label{2thm-eq4}
\big | |b_{n+1}|-|b_n| \big | =\big | (n+1)|a_{n+1}|-n|a_n| \big |  =(n+1)\left ||a_{n+1}|-\cfrac{n}{n+1}|a_n|\right | \leq n+1
\end{equation}
which implies that
 $$|a_{n+1}|-|a_n|\leq\left | |a_{n+1}|-\cfrac{n}{n+1}|a_n| \right | \leq 1,
$$
and the proof is complete.
\hfill{$\Box$}

%%%%%%%%%%%%%%%%%%%%%%%%%%%%%%%%%%%%%%%%%5

\begin{example} Consider the function $f$ defined by \eqref{2condirection}, namely,
$$
f(z)=\cfrac{z-z^2/2}{(1-z)^2} =\sum_{n=1}^{\infty} \frac{n+1}{2}z^n.
$$
It is easy to check that $f$ satisfies the hypothesis of Theorem \ref{2thm4}.  For this function, we have
$$
|a_{n+1}|-|a_n|=\cfrac{n+2}{2}-\cfrac{n+1}{2}=\frac{1}{2}<1.
$$
\end{example}

%%%%%%%%%%%%%%%%%%%%%%%%%%%%%%%%%%%%%%%%%%%%%%%

\begin{example} Consider the function $f$ defined by
$$
f(z)=\cfrac{z}{\sqrt{1-z^2}} =\sum_{n=1}^{\infty} \cfrac{\Gamma(n+1/2)}{\pi \Gamma(n+1)}\,z^{2n+1}.
$$
A simple computation shows that $f  \in \mathcal{C} (-1/2)$ and for this function, we see that
$$
|a_{n+1}|-|a_n|=\cfrac{\Gamma(n+1/2)}{\pi \Gamma(n+1)} <1,
$$
so the result is compatible with Theorem \ref{2thm4}.
\end{example}

%%%%%%%%%%%%%%%%%%%%%%%%%%%%%%%%%%%%%%%%%%%%%%

\subsection{Proof of Theorem \ref{2thm5}}
Let $f\in\mathcal{C} (-1/2)$. Then we have
$$
\big | (k+1)|a_{k+1}|-k|a_k| \big | \leq k+1 \mbox{ for $k\geq 1$},
$$
 by \eqref{2thm-eq4}. Here $a_1=1.$ Using the triangle inequality, we deduce that for $n \geq m$
\begin{align*}
\big | |n|a_n|-m|a_m| \big |  &= \left |\sum_{k=m}^{n-1}(k+1)|a_{k+1}|-k|a_k| \right |\\
&\leq \sum_{k=m}^{n-1}\big | (k+1)|a_{k+1}|-k|a_k| \big |  \\
&\le \sum_{k=m}^{n-1}(k+1)=\cfrac{(n^2-m^2)+(n-m)}{2}.
\end{align*}
Clearly the equality holds for $f \in  \mathcal{C} (-1/2)$  defined by \eqref{2condirection} 
in which the coefficient of $z^n$ is $(n+1)/2.$  \hfill{$\Box$}

%%%%%%%%%%%%%%%%%%%%%%%%%%%%%%%%%%%%

\begin{remark}
It would be interesting to see an improved version of our results in which the upper bounds
are depending upon sharp absolute constant $M$.
\end{remark}

\section*{Acknowledgments}
The authors thank the referee for many useful comments. The  work of the second author is supported 
by Mathematical Research Impact Centric 
Support (MATRICS) of DST, India  (MTR/2017/000367).

\end{document}